\documentclass[12pt]{article}
\usepackage{amsmath,amsthm}
\usepackage{amssymb}
\usepackage{enumerate}
\usepackage{graphicx,tikz}

\usepackage[mathscr]{eucal}
\usepackage[cm]{fullpage}
\usepackage{color}
\usepackage[latin1]{inputenc}

\theoremstyle{plain}
\newtheorem{theorem}{Theorem}[section]
\newtheorem{proposition}[theorem]{Proposition}
\newtheorem{lemma}[theorem]{Lemma}
\newtheorem{corollary}[theorem]{Corollary}

\theoremstyle{remark}

\numberwithin{equation}{section}

\begin{document}

\title{The rank of the inverse semigroup of partial automorphisms on a finite fence}

\author{J. Koppitz and T. Musunthia\footnote{This research is partially supported by the
Thailand Research Fund, Grant No. MRG6180296 and Research Fund of
Faculty of Science, Silpakorn University, Grant No. SRF-JRG-2561-08.}}

\maketitle

\renewcommand{\thefootnote}{}

\footnote{2010 \emph{Mathematics Subject Classification}: 20M18, 20M20}

\footnote{\emph{Key words}: partial injections, finite transformation semigroup, fence, rank,
generators}

\renewcommand{\thefootnote}{\arabic{footnote}}
\setcounter{footnote}{0}

\begin{abstract}
A fence is a particular partial order on a (finite) set, close to
the linear order. In this paper, we calculate the rank of the semigroup $%
\mathcal{FI}_{n}$ of all order-preserving partial injections on an $n$%
-element fence. In particular, we provide a minimal generating set for $%
\mathcal{FI}_{n}$. In the present paper, $n$ is odd since this problem for
even $n$ was already solved by I. Dimitrova and J. Koppitz.
\end{abstract}

\section{Introduction}
Let $n\in \mathbb{N}$ and denote by $\mathcal{PT}_{n}$ the semigroup (under
composition) of all partial transformations on the set $\overline{n}%
:=\{1,\ldots ,n\}$ of the first $n$ natural numbers. The set $\mathcal{I}_{n}
$ of all partial injections on $\overline{n}$ forms an inverse subsemigroup
of $\mathcal{PT}_{n}$. For more information about the symmetric inverse
semigroup $\mathcal{I}_{n}$, we refer the reader to O. Ganyushkin and
V. Mazorchuk's book \cite{Ganyushkin&Mazorchuk:2009}.

Let $\preceq $ be any partial order on $\overline{n}$. The pair $(\overline{n%
},\preceq )$ can be regarded as a digraph. Let $\alpha \in \mathcal{PT}_{n}$.
Then $\alpha $ is called order-preserving on $\overline{n}$ with respect
to $\preceq $ if $a\preceq b\Rightarrow a\alpha \preceq b\alpha $, for all
$a,b\in dom\alpha $. If $\alpha \in \mathcal{I}_{n}$ is order-preserving then it is a
partial injective endomorphisms on the digraph $(\overline{n},\preceq )$. Clearly, the
set $\mathcal{I}End(\overline{n},\preceq )$ of all partial injective endomorphisms on $(%
\overline{n},\preceq )$ forms a submonoid of $\mathcal{I}_{n}$, which has
not to be inverse, in general. A regular element $\alpha $ in $\mathcal{I}End(%
\overline{n},\preceq )$ is characterized by the following property:%
\begin{equation*}
a\preceq b\Leftrightarrow a\alpha \preceq b\alpha, \text{ for all }a,b\in dom\alpha\text{.}
\end{equation*}%
Such regular elements in $\mathcal{I}End(\overline{n},\preceq )$ are called partial
automorphisms on $(\overline{n},\preceq )$. The set $\mathcal{P}Aut(\overline{n},
\preceq)$ of all strong partial automorphisms on $(\overline{n},\preceq )
$ forms an inverse subsemigroup of $\mathcal{I}End(\overline{n},\preceq )$.

A very important particular and natural case occurs when a linear order $\leq $,
e.g. that one induced by the usual order on the natural numbers, is
considered. The monoid $\mathcal{PIO}_{n}$ of all partial order-preserving
injections on $(\overline{n},\leq )$ has been extensively studied. Basic
information about the monoid of $\mathcal{PIO}_{n}$, one can find in
\cite{Fernandes&Gomes&Jesus:2005}. In \cite{Umar:1993}, the author considers
generating sets of the semigroup of all partial
injective decreasing maps of $(\overline{n},\leq )$, a submonoid of $\mathcal{PIO}%
_{n}$. The maximal subsemigroups of the ideals of some semigroups of partial
injections on $(\overline{n},\leq )$ were determined by I. Dimitrova and J. Koppitz \cite{Dimitrova&Koppitz:2009}.
In \cite{Al-Kharousi&Kehinde&Umar:2010}, the authors consider distance-preserving
injections on $(\overline{n},\leq )$. They study the algebraic structure of such semigroups,
e.g. the Green's relations.

A non-linear order close to a linear order in some sense is the so-called
zig-zag order. The pair $(\overline{n},\preceq )$ is called zig-zag poset or
fence if
\begin{eqnarray*}
1 &\prec &2\succ \cdots \prec n-1\succ n\text{ or }1\succ 2\prec \cdots
\succ n-1\prec n\text{ if }n\text{ is odd} \\
\text{and }1 &\prec &2\succ \cdots \succ n-1\prec n\text{ or }1\succ 2\prec
\cdots \prec n-1\succ n\text{ if }n\text{ is even.}
\end{eqnarray*}%
The definition of the partial order $\preceq $ is self-explanatory. Observe
that every element in a fence is either minimal or maximal.

If the domain of an $\alpha \in $ $\mathcal{PT}_{n}$ is $\overline{n}$, i.e.
$dom\alpha =\overline{n}$, then $\alpha $ is called (full) transformation on
$\overline{n}$. The set $\mathcal{T}_{n}$ of all full transformations on $%
\overline{n}$ forms a submonoid of $\mathcal{PT}_{n}$. The monoid $\mathcal{%
TF}_{n}$ of all order-preserving transformations within $\mathcal{T}_{n}$
(with respect to $\preceq $), i.e. of all endomorphisms on $(\overline{n}%
,\preceq )$, was first investigated by J.D. Currie and T.I. Visentin in \cite{Currie&Visentin:1991} and by
A. Rutkowski \cite{Rutkowski:1992}. In \cite{Currie&Visentin:1991}, by using generating functions, the authors calculate
the cardinality of $\mathcal{TF}_{n}$ for the case that $n$ is even. On the
other hand, an exact formula for the number of endomorphisms on $(\overline{n%
},\preceq )$ for even as\ well as\ odd $n$ was given in \cite{Rutkowski:1992}. Recently, in
\cite{Musunthia&Koppitz&Vitor:2017}, the authors determine the rank of $\mathcal{TF}_{n}$. Recall that
the rank of a semigroup $S$, denoted by $\text{rank}S$, is the minimal size of a
generating set of $S,$%
\begin{equation*}
\text{rank}S:=\min \{\left\vert A\right\vert :A\subseteq S\text{, }A\text{
generates }S\}.
\end{equation*}%
In particular, a concrete generating set of $\mathcal{TF}_{n}$ of minimal
size is given in \cite{Musunthia&Koppitz&Vitor:2017}. Moreover, the authors characterize the
transformations on $\overline{n}$ preserving the fence. It is worth
mentioning that several other properties of monoids of order-preserving
transformations of a fence were also studied. In \cite{Jitman&Srithus&Worawannotai:2018,Tanyawong&Srithus&Chinram:2016}, the authors discussed
the regular elements of these monoids. Coregular elements (i.e. elements $%
\alpha $ with the property $\alpha =\alpha ^{3}$) of these monoids were
determined in \cite{Jendana&Srithus:2015}. Some relative ranks of the monoid of all partial
transformations preserving an infinite zig-zag order were determined in
\cite{Dimitrova&Koppitz&Lohapan:2017}.

In this paper, we will denote the semigroup of all partial
automorphisms on $(\overline{n},\preceq )$ by $\mathcal{FI}_{n}$, i.e. $\mathcal{P}Aut(\overline{n},
\preceq )=\mathcal{FI}_{n}$. This inverse semigroup was
first studied by I. Dimitrova and J. Koppitz in \cite{Dimitrova&Koppitz:2017}. The authors described the
Green's relations on $\mathcal{FI}_{n}$. In fact, they described only the $%
\mathcal{J}$- relation since the relations $\mathcal{L},\mathcal{R}$, and $%
\mathcal{H}$ are clear because $\mathcal{FI}_{n}$ is an inverse semigroup.
Moreover, they show that $\mathcal{FI}_{n}$ is generated by the set
\begin{equation*}
J_{n}:=\{\alpha \in \mathcal{FI}_{n}:\text{rank}\alpha \geq n-2\}\text{.}
\end{equation*}%
Recall that the rank of a (partial) transformation $\alpha $ (in symbols: $%
\text{rank}\alpha $) is the size of the range of $\alpha $ (in symbols: $\text{im}\alpha $%
), i.e. $\text{rank}\alpha =\left\vert \text{im}\alpha \right\vert $. For the case that $n$
is even, it is proved that $\text{rank}\mathcal{FI}_{n}=n+1$ and a concrete
generating set of $\mathcal{FI}_{n}$ with $n+1$ elements is given in \cite{Dimitrova&Koppitz:2017}. On the
other hand, the rank of $\mathcal{FI}_{n}$ is still an open problem,
whenever $n$ is odd. We will solve it in the present paper. We will determine
the rank of $\mathcal{FI}_{n}$ and give a concrete generating set of $\mathcal{FI}_{n}$
with minimal size in the case that $n$ is odd.

Without loss of generality, let $1\prec 2\succ 3\prec \cdots \succ n$. Such
fences are also called up-fences. The fence $1\succ 2\prec 3\succ \cdots
\prec n$ would be called down-fence. We avoid both notations up-fence and
down-fence. In fact, in order to check a fence is an up-fence or down-fence,
we need that $1$ and $2$ are comparable with respect to $\preceq $. Recall
that $x,y\in \overline{n}$ are comparable with respect to $\preceq$ if $x\prec y$ or $x=y$ or $x\succ
y $. Otherwise, $x$ and $y$ are called incomparable. But the restriction that
$1 $ and $2$ belong to the fence and are comparable is an unnecessary
restriction for the concept fence since instead of $\overline{n}$ one could
choose another $n$-element set or one could define $\preceq $ on $\overline{n%
}$ such that $1$ and $2$ are incomparable.

But if the fence $(\overline{n},\preceq )$ is defined as above (it is the most
natural way to do it in that way) then we observe that any $x,y\in
\overline{n}$ are comparable if and only if $x\in \{y-1,y,y+1\}$. For general
background on Semigroup Theory and standard notation, we refer the reader to
Howie`s book \cite{Howie:1995}.

\section{Main result}
Let us fix now an odd natural number $n$. Clearly, for $U\subseteq \overline{%
n}$, the partial identity mapping $id_{U}:=id_{\overline{n}}|_{U}$, i.e. the
identity mapping $id_{\overline{n}}$ on $\overline{n}$ restricted to $U$,
is a partial automorphism on $(\overline{n},\preceq )$. In
particular, $id_{\overline{n}}\in \mathcal{FI}_{n}$. There are exactly two
automorphisms in \ $\mathcal{FI}_{n}$, besides $id_{\overline{n}}$ the
reflection
\begin{equation*}
\gamma _{n}:=\left(
\begin{array}{ccccc}
1 & 2 & \cdots & n-1 & n \\
n & n-1 & \cdots & 2 & 1%
\end{array}%
\right) \text{.}
\end{equation*}%
Note that $\mathcal{FI}_{1}$ consists of the identity mapping on $\{1\}$ and
the empty transformation $\emptyset $. Since these both partial
transformations do not generate each other, the rank of $\mathcal{FI}_{1}$
is $2$. We suppose now that $n\geq 3$. I. Dimitrova and J. Koppitz proved $\mathcal{%
FI}_{m}=\left\langle J_{m}\right\rangle $ for all natural numbers $m$ in \cite{Dimitrova&Koppitz:2017}, which comprises several pages and a few
of lemmas in \cite{Dimitrova&Koppitz:2017}. For the case that $n$ is odd, one can shorten the proof.
Therefore, and for the sake of completeness, we will give a new proof for
the particular case that $n$ is odd. For this, we define a series of partial transformations of $J_{n}$. Let

$\alpha _{i}=\left(
\begin{array}{ccccccc}
1 & \cdots  & i-1 & i & i+1 & \cdots  & n \\
1 & \cdots  & i-1 &  & n & \cdots  & i+1%
\end{array}%
\right)$ for even $i\in \{2,\ldots ,n-1\}$,

$a_{i}=id_{\overline{n}\setminus\{i\}}$ for odd $i\in \overline{n}$,

$\beta _{2}^{odd} =\left(
\begin{array}{cccccc}
1 & 2 & 3 & 4 & \cdots  & n \\
2 &  &  & 4 & \cdots  & n%
\end{array}%
\right) \text{, }\beta _{n-1}^{odd}=\left(
\begin{array}{cccccc}
1 & 2 & 3 & \cdots  & n-1 & n \\
n-1 &  & 1 & \cdots  & n-3 &
\end{array}
\right)$,

$\beta _{2}^{even} =\left(
\begin{array}{cccccc}
1 & 2 & 3 & 4 & \cdots  & n \\
& 1 &  & 4 & \cdots  & n%
\end{array}%
\right) \text{, and }\beta _{n-1}^{even}=\left(
\begin{array}{cccccc}
1 & \cdots  & n-3 & n-2 & n-1 & n \\
3 & \cdots  & n-1 &  & 1 &
\end{array}%
\right)$.

In the case $n\geq 5$, we define%

$\alpha _{i,j}=\left(
\begin{array}{ccccccccccc}
1 & \cdots  & i-1 & i & i+1 & \cdots  & j-1 & j & j+1 & \cdots  & n \\
1 & \cdots  & i-1 &  & j-1 & \cdots  & i+1 &  & j+1 & \cdots  & n%
\end{array}%
\right)$  for $2\leq i<j\leq n-1$, where $i$ and $j$ have the same parity,

$\alpha _{1,j}=\left(
\begin{array}{cccccccc}
1 & 2 & \cdots  & j-1 & j & j+1 & \cdots  & n \\
& j-1 & \cdots  & 2 &  & j+1 & \cdots  & n%
\end{array}%
\right)$  and

$\alpha _{j,n}=\left(
\begin{array}{cccccccc}
1 & \cdots  & j-1 & j & j+1 & \cdots  & n-1 & n \\
1 & \cdots  & j-1 &  & n-1 & \cdots  & j+1 &
\end{array}%
\right)$ for odd $j\in \{3,\ldots ,n-2\}$, and

$\alpha _{1,n}=\left(
\begin{array}{ccccc}
1 & 2 & \cdots  & n-1 & n \\
& n-1 & \cdots  & 2 &
\end{array}%
\right)$.

In the case $n\geq 7$, we define%

$\beta _{i}^{odd}=\left(
\begin{array}{ccccccccc}
1 & 2 & 3 & \cdots  & i & i+1 & i+2 & \cdots  & n \\
i &  & 1 & \cdots  & i-2 &  & i+2 & \cdots  & n%
\end{array}%
\right)$ and

$\beta _{i}^{even}=\left(
\begin{array}{ccccccccc}
1 & \cdots  & i-2 & i-1 & i & i+1 & i+2 & \cdots  & n \\
3 & \cdots  & i &  & 1 &  & i+2 & \cdots  & n%
\end{array}%
\right)$ for even $i\in \{4,\ldots ,n-3\}$.

Note that $\beta
_{i}^{even}\beta _{i}^{odd}=id_{\overline{n}\setminus \{i-1,i+1\}}$ for each
even $i\in \{2,\ldots ,n-1\}$, $\alpha _{i}\alpha _{i}=id_{\overline{n}%
\setminus \{i\}}$ for each $i\in \overline{n}$, and $\alpha _{i,j}\alpha
_{i,j}=id_{\overline{n}\setminus \{i,j\}}$ for all $i<j\in \overline{n}$
having the same parity. Further, let $Par_{n}$ be the set of all $\delta
\in \mathcal{FI}_{n}$ such that $x$ and $x\delta$ have different parity for
some $x\in dom\delta$. First, we observe that each $\delta \in Par_{n}$ is
generated by elements of $\mathcal{FI}_{n}\setminus Par_{n}$ and $\{id_{\overline{n}}\}\cup \{\beta _{i}^{odd},\beta
_{i}^{even}:i\in \{2,4,\ldots ,n-1\}\}\subseteq J_{n}$.

\begin{lemma}\label{lmm1}
Let $\delta \in Par_{n}$. Then there are $\beta \in \mathcal{FI}%
_{n}\setminus Par_{n}$ and $l_{1},\ldots ,l_{p},r_{1}\ldots ,r_{p}\in \{id_{\overline{n}}\}\cup \{\beta
_{i}^{odd},\beta _{i}^{even}:i\in \{2,4,\ldots ,n-1\}\}$
with $\delta =l_{1}\cdots l_{p}\beta r_{1}\cdots r_{p}$.
\end{lemma}

\begin{proof}
Let $x\in dom\delta $ such that $x$ and $x\delta $ have different parity. If
$x$ is odd then $x\delta $ is even$,$ where $x\delta -1,x\delta +1\notin
\text{im}\delta $. Let $\beta :=id_{\overline{n}}\delta \beta _{x\delta }^{even}$. Then it is easy
to verify that $\left\vert \left\{ w\in dom\delta :w\text{ and }w\delta
\text{ have different parity}\right\} \right\vert >\left\vert \left\{ w\in
dom\beta :w\text{ and }w\beta \text{ have different parity}\right\}
\right\vert $ and $\beta \beta _{x\delta }^{odd}=\delta \beta _{x\delta }^{even}\beta
_{x\delta }^{odd}=\delta$ since $\text{im}\delta \subseteq dom\beta
_{x\delta }^{even}$. If $x$ is even then $x\delta $ is odd and $%
x-1,x+1\notin dom\delta $. In this case, let $\beta :=\beta _{x}^{odd}\delta id_{\overline{n}}$
and we observe that $\left\vert \left\{ w\in dom\delta :w\text{ and }w\delta
\text{ have different parity}\right\} \right\vert >\left\vert \left\{ w\in
dom\beta :w\text{ and }w\beta \text{ have different parity}\right\}
\right\vert $, where $\beta _{x}^{even}\beta =\beta _{x}^{even}\beta _{x}^{odd}\delta
=\delta$ since $\text{im}\beta _{x}^{odd}\supseteq dom\delta $. \
Now, we can consider $\beta $ instead of $\delta $. Since the domain of $%
\delta $ is finite, we obtain successively after $p$ steps $%
l_{1},\ldots ,l_{p},r_{1}\ldots ,r_{p}\in \{id_{\overline{n}}\}\cup \{\beta
_{i}^{odd},\beta _{i}^{even}:i\in \{2,4,\ldots ,n-1\}\}$ and a $\widetilde{\beta }%
\in \mathcal{FI}_{n}\setminus Par_{n}$ such that $\delta =l_{1}\cdots l_{p}%
\widetilde{\beta }r_{1}\cdots r_{p}$.
\end{proof}

The following fact will be used frequently without to refer it.
If $U$ is a convex subset of the domain of a $\delta \in \mathcal{FI}_{n}$ then
$U\delta =\{x\delta :x\in U\}$ is also a convex set. Next, we show that any $\delta \in \mathcal{FI}_{n}\setminus J_{n}$ with a
convex domain can be generated by elements of $J_{n}$ and a
transformation $\beta \in \mathcal{FI}_{n}$ with $\text{rank}\beta >\text{rank}\delta $.

\begin{lemma}\label{lmm2}
Let $\delta \in \mathcal{FI}_{n}\setminus J_{n}$. If $dom\delta$ is a convex
set then there are $\beta \in \mathcal{FI}_{n}$ with $\text{rank}\beta =\text{rank}\delta
+1$ and $\alpha \in J_{n}$ such that $\delta =\alpha \beta $.
\end{lemma}

\begin{proof}
Suppose that $dom\delta $ is a convex set.

Since both intervals $dom\delta$ and $\text{im}\delta$ have a lenght less than or equal
$n-3$, there are $x,w\in\bar{n}$ such that $w-1,w,w+1\in\{0,1,\dots,n,n+1\}\setminus
dom\delta$ and $x-1,x,x+1\in\{0,1,\dots,n,n+1\}\setminus \text{im}\delta$.  We define a
transformation $\beta $ by
\begin{equation*}
r\beta :=r\delta \text{ for all }r\in dom\delta \text{ and }w\beta :=x\text{.%
}
\end{equation*}%
Clearly, $\beta \in \mathcal{FI}_{n}$ with $\text{rank}\beta =\text{rank}\delta +1$ and $\delta =id_{\overline{n}\setminus \{w\}}\beta $.
\end{proof}

We note that the empty set is convex. So, the empty
transformation $\emptyset $ is a product of two transformations with rank $%
\geq 1$. Now, we can prove the following proposition, which is a particular case of
Theorem 3.15 in \cite{Dimitrova&Koppitz:2017} for $n$ is odd.

\begin{proposition}\label{prop3}
$\mathcal{FI}_{n}$ is generated by $J_{n}$.
\end{proposition}

\begin{proof}
Let $3\leq r\leq n$ and suppose that $\beta \in \left\langle
J_{n}\right\rangle $ for all $\beta \in \mathcal{FI}_{n}$ with $\text{rank}\beta
>n-r$. Let now $\delta \in \mathcal{FI}_{n}$ with $\text{rank}\delta =n-r$. By
Lemma \ref{lmm1}, we can restrict us to the case $\delta \in \mathcal{FI}%
_{n}\setminus Par_{n}$. Now, we choose $x\in \overline{n}\setminus dom\delta $
in the following way: There is $z<x$ such that $w\notin dom\delta $ for all $%
w<z$ and $\{z,\ldots ,x-1\}\subseteq dom\delta $. Clearly, if $dom\delta
\neq \{r+1,\ldots ,n\}$ then such $x$ exists. Note, by Lemma \ref{lmm2}, we can skip
the case that $dom\delta =\{r+1,\ldots ,n\}$.

Now, we will show that one can restrict oneself to the case that $w\delta
>\{z,\ldots ,x-1\}\delta$ for all $w\in dom\delta \setminus \{z,\ldots ,x-1\}$ and $%
(x-2)\delta <(x-1)\delta $, whenever $x-2\in dom\delta $. For this let $%
\{r_{1}<\cdots <r_{s}\}$ be a convex interval in $dom\delta $ with $%
r_{1}-1,r_{s}+1\in \{0,1,\ldots ,n,n+1\}\setminus dom\delta $ such that $%
w\delta >\{r_{1},\ldots ,r_{s}\}\delta$ for all $w\in dom\delta \setminus \{r_{1},\ldots
,r_{s}\}$. Suppose that $x-1\notin \{r_{1},\ldots ,r_{s}\}$, i.e. $r_{1}>2$. If $r_{s}$ is
odd then we define $\sigma :=\gamma _{n}\alpha _{n-r_{s}}\gamma _{n}\delta $%
, where $\delta =\gamma _{n}\alpha _{n-r_{s}}\gamma _{n}\gamma _{n}\alpha
_{n-r_{s}}\gamma _{n}\delta $ and $\alpha _{0}=\gamma _{n}$. If $r_{s}$ is even and $r_{1}$ is odd then we
define $\sigma :=\alpha _{r_{1}-1}\delta \gamma _{n}$, where $\delta =\alpha
_{r_{1}-1}\alpha _{r_{1}-1}\delta \gamma _{n}\gamma _{n}$. If both $r_{1}$
and $r_{s}$ are even and $p:=\max \{z\delta ,\ldots ,(x-1)\delta \}$ is odd
then we define $\sigma :=\delta \gamma _{n}\alpha _{n-p}\gamma _{n}$, where $%
\delta =\delta \gamma _{n}\alpha _{n-p}\gamma _{n}\gamma _{n}\alpha
_{n-p}\gamma _{n}$. If all $r_{1}$, $r_{s}$, and $p$ are even then $1\notin im\delta$ and we define
$\sigma :=\delta \alpha _{1,p+1}$, where $\delta =\delta \alpha _{1,p+1}\alpha _{1,p+1}$. Hence, $\delta $ is the product of $%
\sigma $ and transformations in $J_{n}$. Straightforward calculations show
that there are $\widetilde{z}<\widetilde{x}\in \overline{n}$ such that $\{\widetilde{z}%
,\ldots ,\widetilde{x}-1\}\subseteq dom\sigma $ is a convex set and $%
\widetilde{x},w\notin dom\sigma $ for all $w<\widetilde{z}$, where $w\sigma
>\{\widetilde{z},\ldots ,\widetilde{x}-1\}\sigma $ for all $w\in dom\sigma
\setminus \{\widetilde{z},\ldots ,\widetilde{x}-1\}$. Suppose that $%
\widetilde{x}-2\in dom\sigma $ with $(\tilde{x}-2)\sigma >(\tilde{x}-1)\sigma $. If $%
\widetilde{z}$ is odd then we define $\upsilon :=\sigma \gamma _{n}\alpha
_{n-\widetilde{z}\delta }\gamma _{n}$, where $\sigma =\sigma \gamma
_{n}\alpha _{n-\widetilde{z}\delta }\gamma _{n}\gamma _{n}\alpha _{n-%
\widetilde{z}\delta }\gamma _{n}$. If $\widetilde{z}$ is even and $%
\widetilde{x}-1$ is odd then we define $\upsilon :=\gamma _{n}\alpha _{n-%
\widetilde{x}+1}\gamma _{n}\sigma $, where $\sigma =\gamma _{n}\alpha _{n-%
\widetilde{x}+1}\gamma _{n}\gamma _{n}\alpha _{n-\widetilde{x}+1}\gamma
_{n}\sigma $. If both $\widetilde{z}$ and $\widetilde{x}-1$ are even then we
define $\upsilon :=\alpha _{\widetilde{z}-1,\widetilde{x}}\sigma $, where $%
\sigma =$ $\alpha _{\widetilde{z}-1,\widetilde{x}}\alpha _{\widetilde{z}-1,%
\widetilde{x}}\sigma $. Hence, $\sigma $ is the product of $\upsilon $ and
transformations in $J_{n}$. Simple calculations show that there are $%
\widehat{z}<\widehat{x}\in \overline{n}$ such that $\{\widehat{z},\ldots ,\widehat{x}%
-1\}\subseteq dom\upsilon $ is a convex set and $\widehat{x},w\notin
dom\upsilon $ for all $w<\widehat{z}$, where $w\upsilon >\{\widehat{z}%
,\ldots ,\widehat{x}-1\}\upsilon $ for all $w\in dom\upsilon \setminus \{%
\widehat{z},\ldots ,\widehat{x}-1\}$ and $(\widehat{x}-2)\upsilon <(\widehat{%
x}-1)\upsilon $, whenever $\widetilde{x}-2\in dom\nu $. This proves the desired restriction.

Without loss of generality, we can assume that $z\delta \in
\{1,2\}$. Otherwise, we consider the transformation $\widehat{\delta }%
:=\delta \left(
\begin{array}{cccc}
3 & 4 & \cdots & n \\
1 & 2 & \cdots & n-2%
\end{array}%
\right) ^{\left\lfloor \frac{z\delta -1}{2}\right\rfloor }$, where $\delta =%
\widehat{\delta }\left(
\begin{array}{cccc}
1 & 2 & \cdots & n-2 \\
3 & 4 & \cdots & n%
\end{array}%
\right) ^{\left\lfloor \frac{z\delta -1}{2}\right\rfloor }$ and $z\widehat{%
\delta }\in \{1,2\}$. We put now
\begin{equation*}
a:=(x-1)\delta \text{.}
\end{equation*}%
Since $x\notin dom\delta $ and $w\delta >(x-1)\delta $ for all $w\in
dom\delta \setminus \{z,\ldots ,x-1\}$, either $a+1\notin \text{im}\delta $ or $dom\delta $
is a convex set. In the latter case, by Lemma \ref{lmm2}, there are $\beta \in \mathcal{FI}_{n}$ with $%
\text{rank}\beta =\text{rank}\delta +1$ and $\alpha \in J_{n}$ such that $\delta =\alpha \beta $.
Since $\beta \in \left\langle J_{n}\right\rangle $ by the inductive assumption, we obtain $
\delta =\alpha\beta \in \left\langle J_{n}\right\rangle $. So, we have to
consider the case $a+1\notin \text{im}\delta $, where either $x+1\in dom\delta $ or
$x+1\notin dom\delta .$

Suppose that $x+1\in dom\delta $. If $(x+1)\delta =a+2$ then we define a partial
transformation $\beta _{1}$ by
\begin{equation*}
w\beta _{1}:=w\delta \text{ for all }w\in dom\delta \text{ and }x\beta
_{1}:=a+1\text{.}
\end{equation*}%
Clearly, $\beta _{1}\in \mathcal{FI}_{n}$ with $\text{rank}\beta _{1}>n-r.$ We have
$\delta =\alpha _{a}\beta _{1}\in \left\langle J_{n}\right\rangle $. Now, we
show that there is $\widehat{\delta }\in \mathcal{FI}_{n}$ with $w\widehat{%
\delta }=w\delta $ for $w\in dom\delta \cap \{1,\ldots ,x-1\}$ and $(x+1)%
\widehat{\delta }=a+2$ such that $\delta \in \langle J_{n},\widehat{%
\delta }\rangle $. If it is the case then we have $\delta \in
\left\langle J_{n}\right\rangle $ since $\widehat{\delta }\in \left\langle
J_{n}\right\rangle $ by the previous considerations.

Suppose that $a+2\in \text{im}\delta $. Then there is $y\in dom\delta $ with $%
y\delta =a+2$. Since $w\delta \leq a$ for all $w\in dom\delta $ with $w<x$, we
conclude that $y\geq x+1$. If $y+1\notin dom\delta $ then we put $\beta
_{2}:=\alpha _{x,y+1}$, whenever $y<n$ and $\beta _{2}:=\alpha _{x}$,
whenever $y=n$. Clearly, $\beta _{2}\in \mathcal{FI}_{n}$. Then $\delta =\beta _{2}\beta _{2}\delta $, where $w\beta
_{2}\delta =w\delta $ for $w\in dom\delta \cap \{1,\ldots ,x-1\}$ and $%
(x+1)\beta _{2}\delta =a+2$. Now, we consider the case that $y+1\in dom\delta
$. This provides $(y+1)\delta =a+3$ and $y-1\notin dom\delta $ since $%
a+1\notin \text{im}\delta $. If $y-1$ is even then $\delta =\alpha _{y-1}\alpha
_{y-1}\delta $, where $w\alpha _{y-1}\delta =w\delta $ for $w\in dom\delta
\cap \{1,\ldots ,x-1\}$, $n\alpha _{y-1}\delta =a+2$, and $n+1\notin
dom(\alpha _{y-1}\delta) $ (which provides a previous case). If $y-1$ is odd
and there is an odd \ $u>y+1$ with $u\notin dom\delta $ then we have $\delta
=\alpha _{y-1,u}\alpha _{y-1,u}\delta $, where $w\alpha _{y-1,u}\delta
=w\delta $ for $w\in dom\delta \cap \{1,\ldots ,x-1\}$, $(u-1)\alpha
_{y-1,u}\delta =a+2$, and $u\notin dom(\alpha _{y-1,u}\delta) $ (which provides
a previous case). Let now $y-1$ be odd and let each $u\in \overline{n}%
\setminus dom\delta $ with $u>y+1$ be even. Then we observe that $w-1,w+1\in
\{y,\ldots ,n\}\delta $, for each even $w\in \{y,\ldots ,n\}\delta $. Hence,
there is an odd $k\in \{y,\ldots ,n\}\delta $ with $k+1\notin \{y,\ldots
,n\}\delta $. Because $(x+1)\delta $ ($>y\delta $) is even,
this implies the existence of an even $u<y$ with $u\notin dom\delta $.

If $u>x$ then we have $\delta =\alpha _{u}\alpha _{u}\delta $, where $%
w\alpha _{u}\delta =w\delta $ for $w\in dom\delta \cap \{1,\ldots ,u-1\}$
with $u>x$ and $(n-y+1)\alpha _{w}\delta =a+2$, where $(n-y+1+1)\notin
dom(\alpha _{u}\delta) $. We have again a previous case. Finally, we have to
consider the case $u<x$. In particular, $1,2\notin
dom\delta $. We have $\alpha _{2}\gamma _{n}\gamma _{n}\alpha _{2}\delta
=\delta $, where $w\gamma _{n}\alpha _{2}\delta =(w+2)\delta $, for all $w\in
\{1,\ldots ,n-2\}$ with $w+2\in dom\delta $. Then $(y-2)\gamma
_{n}\alpha _{2}\delta =a+1$, $y-3$ is odd and $n\notin dom(\gamma _{n}\alpha
_{2}\delta) $, i.e. $\gamma _{n}\alpha _{2}\delta $ is a previous case.

Admit now $a+2\notin \text{im}\delta $. If $(x+1)\delta +1\notin \text{im}\delta $ then we
have $\delta =\delta \alpha _{a+1,(x+1)\delta +1}\alpha _{a+1,(x+1)\delta +1}
$, where $w\delta \alpha _{a+1,(x+1)\delta +1}=w\delta $, for $w\in
dom\delta \cap \{1,\ldots ,x-1\}$, and $(x+1)\delta \alpha
_{a+1,(x+1)\delta +1}=a+2$. Suppose now that $(x+1)\delta +1\in \text{im}\delta $, i.e. $(x+1)\delta -1\notin im\delta $.
If $(x+1)\delta $ is even then $a+2$ is even and we put $\beta _{3}:=\alpha
_{a+2}.$ If $(x+1)\delta $ is odd then we put $\beta _{3}:=\alpha _{(x+1)\delta
-1}$. Clearly, $\beta _{3}\in \mathcal{FI}_{n}$, $\delta =\delta \beta _{3}\beta _{3}$, where $w\delta \beta
_{3}=w\delta $ for $w\in dom\delta \cap \{1,\ldots ,x-1\}$, and $a+2\notin im\delta \beta _{3}$. Since $%
(x+1)\delta -1\notin \text{im}\delta $, we can conclude that $(x+1)\delta \beta
_{3}+ 1\notin \text{im}\delta \beta _{3}$ (the previous case for $\delta \beta _{3}$).

We finish the proof with the case that $x+1\notin dom\delta $. If $a+2\notin
\text{im}\delta $ we define a partial transformation $\beta _{4}$ by
\begin{equation*}
w\beta _{4}:=w\delta \text{ for }w\in dom\delta \text{ and }x\beta _{4}:=a+1%
\text{.}
\end{equation*}%
It is easy to verify that $\beta _{4}\in
\mathcal{FI}_{n}$. Then we have $\delta =id_{\overline{n}
\setminus \{x\}}\beta _{4}$. Suppose that there is $y\in
dom\delta $ with $y\delta =a+2$. We can assume that $y+1\notin dom\delta $.
Otherwise, $y-1\notin dom\delta $ and there is $i\in \{0,1\}$ such that $x+i$
is even and $\delta =\alpha _{x+i}\alpha _{x+i}\delta $, where $w\alpha
_{x+i}\delta =w\delta $ for $w\in dom\delta \cap \{1,\ldots ,x-1\}$. In
particular, we have $(n-y+x+i+1)\alpha _{x+i}\delta =a+2$ and $(n-y+x+i+2)\notin
dom(\alpha _{x+i}\delta) $. Here we have a previous case (for $\alpha
_{x+i}\delta $). Having $y+1\notin dom\delta $, we obtain $\delta =\alpha
_{x,y+1}\alpha _{x,y+1}\delta $, where $w\alpha _{x,y+1}\delta =w\delta $
for $1\leq w\leq x$ and $(x+1)\alpha _{x,y+1}\delta =a+2$, a previous case
for the transformation $\alpha _{x,y+1}\delta $.

Altogether, we have shown that there is $\widehat{\delta }\in \mathcal{FI}%
_{n}$ with $w\widehat{\delta }=w\delta $ for $w\in dom\delta \cap \{1,\ldots
,x-1\}$ and $(x+1)\widehat{\delta }=a+2$ such that $\delta \in \langle
J_{n},\widehat{\delta }\rangle$. Consequently, we have shown that
$\mathcal{FI}_{n}\subseteq \left\langle J_{n}\right\rangle $.
\end{proof}

Now, we construct a generating set of $\mathcal{FI}_{n}$ of minimal size.
By Proposition \ref{prop3} and since no element in $J_{n}$
can be generated by elements which not belong to $J_{n}$, we have to find a
generating set of $J_{n}$ of minimal size. For this, we define $G_{3}:=\{\gamma _{3},\alpha _{1},\alpha _{2},\beta
_{2}^{odd},\beta _{2}^{even}\}$ and $G_{n}:=\{\gamma _{n}\}\cup \{\alpha
_{i}:i\in \{1,,\ldots ,\frac{n+1}{2}\}$ is odd$\}\cup \{\alpha _{i}:i\in
\{2,4,\ldots ,n-3\}\}\cup \{\beta _{i}^{odd},\beta _{i}^{even}:i\in
\{2,\ldots ,\frac{n+1}{2}\}$ is even$\}\cup \{\alpha _{i,j}:i,j\in \overline{%
n}$ are odd with $4\leq j-i<n-1$, $i\leq n-j+1\}$, whenever $n\geq 5$.

\begin{lemma}\label{lmm(bf4)}
Let $\delta \in J_{n}\cap Par_{n}$. Then there is
exactly one $x\in dom\delta $ such that $x$ and $x\delta $ have different
parity. In particular, it holds $x\delta \in \{1,n\}$ or $x\in \{1,n\}$.
\end{lemma}

\begin{proof}
Since $\delta \in Par_{n}$, there is an $x\in dom\delta $ such that $x$
and $x\delta $ have different parity. Assume that $x$ is even and $x\delta \notin \{1,n\}$. Then $x-1,x+1\notin
dom\delta $ and $x\delta -1,x\delta +1\notin \text{im}\delta $. Hence, since $rank\delta \geq n-2$, there is
only one $y\in dom\delta $ such that $y$ and $y\delta $ have different
parity, namely $x$. So, we have $\frac{n+1}{2}-2$ odd elements in the domain
of $\delta $, which have to be mapped to $\frac{n+1}{2}-1$ odd elements in
the image of $\delta $, a contradiction. By the same argumentation, we
obtain that $x$ is unique with $x\in \{1,n\}$, whenever $x\delta $ is even.
\end{proof}

\begin{lemma}\label{lmm4}
Let $\delta \in J_{n}$. Then $\delta \in \left\langle G_{n}\right\rangle $.
\end{lemma}

\begin{proof}
If $\text{rank}\delta =n$ then $\delta =\gamma _{n}\in G_{n}$ or $\delta =id_{\overline{n}}=\gamma
_{n}\gamma _{n}$.

Let now $\text{rank}\delta =n-1$. Then there is $i\in \overline{n}\setminus
dom\delta $. We observe that $im\delta =dom\delta $ or $im\delta =(dom\delta
)\gamma _{n}$.

If $i\leq n-3$ is even then $\delta =\alpha _{i}\in G_{n}$ or $\delta
=\alpha _{i}^{2}$ or $\delta =\alpha _{i}\gamma _{n}$ or $\delta =\alpha
_{i}^{2}\gamma _{n}$ or

$\delta =\left(
\begin{array}{ccccccc}
1 & \cdots  & i-1 & i & i+1 & \cdots  & n \\
i-1 & \cdots  & 1 &  & i+1 & \cdots  & n%
\end{array}%
\right) =\gamma _{n}\alpha _{n-i+1}\gamma _{n}$ or $\delta =\gamma
_{n}\alpha _{n-i+1}$ or

$\delta =\left(
\begin{array}{ccccccc}
1 & \cdots  & i-1 & i & i+1 & \cdots  & n \\
i-1 & \cdots  & 1 &  & n & \cdots  & i+1%
\end{array}%
\right) =\gamma _{n}\alpha _{n-i+1}\gamma _{n}\alpha _{i}$ or $\delta
=\gamma _{n}\alpha _{n-i+1}\gamma _{n}\alpha _{i}\gamma _{n}$.

If $i=n-1$ \ then $\delta =\gamma _{n}\alpha _{_{2}}$ or $\delta =\gamma _{n}\alpha _{2}\gamma _{n}$ or $\delta
=\gamma _{n}\alpha _{2}^{2}$ or $\delta =\gamma _{n}\alpha _{2}^{2}\gamma
_{n}$.

If $i$ is odd and $i\leq \frac{n+1}{2}$ then $\delta =\alpha _{i}\in G_{n}$
or $\delta =\alpha _{i}\gamma _{n}$ and if $i>\frac{n+1}{2}$ then $\delta=\gamma
_{n}\alpha _{n-i+1}$ or $\delta=\gamma _{n}\alpha _{n-i+1}\gamma _{n}$, where $-i<
-\frac{n+1}{2}$, i.e. $n-i+1<\frac{n+1}{2}$.

Now, we consider the case $\text{rank}\delta =n-2$ and $\delta \notin Par_{n}$. Then
there are $i,j\in \overline{n}\setminus dom\delta $ with $i<j$. Let both $i$
and $j$ be odd. If $4\leq j-i<n-1$ and $i\leq n-j+1$ then $\delta =$ $%
\alpha _{i,j}\in G_{n}$ or $\delta =\alpha _{i,j}\gamma _{n}$. Suppose now that $j-i=n-1$, i.e. $i=1$ and $j=n$.
Here, we obtain $\delta =\alpha _{1,n}=\gamma _{n}\alpha _{1}\alpha _{n}$ or $\delta =\alpha _{1,n}^{2}=\alpha _{1}\alpha _{n}$.
If $2=j-i<n-1$, i.e. $i+2=j$ then $\delta =\left(
\begin{array}{ccccccccc}
1 & \cdots  & i-1 & i & i+1 & i+2 & i+3 & \cdots  & n \\
1 & \cdots  & i-1 &  & i+1 &  & i+3 & \cdots  & n%
\end{array}%
\right) =\alpha _{i}\alpha _{i+2}$ or $\delta =\left(
\begin{array}{ccccccccc}
1 & \cdots  & i-1 & i & i+1 & i+2 & i+3 & \cdots  & n \\
n & \cdots  & n-i+2 &  & n-i &  & n-i-2 & \cdots  & 1%
\end{array}%
\right) =\alpha _{i}\alpha _{i+2}\gamma _{n}$. Suppose now $i>n-j+1$. Then
we put $r:=n-j+1$ and $s:=n-i+1$. We observe that $%
r=n-j+1<i=n-(n-i+1)+1=n-s+1$ and $\delta =\gamma _{n}\alpha _{r,s\text{.}%
}\gamma _{n}$.
Admit now that $i$ is even or $j$ is even. The following is not
hard to verify. If $i$ is even and $j$ is odd then $\delta \in \{\alpha
_{i},\alpha _{i}^{2},\gamma _{n}\alpha _{n-i+1}\gamma _{n},\alpha _{i}\gamma
_{n},\alpha _{i}^{2}\gamma _{n},\gamma _{n}\alpha _{n-i+1},\gamma _{n}\alpha
_{n-i+1}\gamma _{n}\alpha _{i}, \gamma _{n}\alpha _{n-i+1}\gamma _{n}\alpha _{i}\gamma _{n}\}\{\alpha _{j},\alpha _{j}\gamma _{n}\}$. If $%
i$ is odd and $j$ is even then we obtain $\delta \in \left\langle
G_{n}\right\rangle $, dually. If both $i$ and $j$ are even then we have $%
\alpha _{i,j}=\alpha _{i}\alpha _{j-i}\alpha _{i}$ and $\delta \in \Gamma
_{1}\Gamma _{2}\{\gamma _{n},id_{\overline{n%
}}\}$, where $\Gamma _{1}=\{id_{\overline{n}},\alpha
_{j}\}\{id_{\overline{n}},\alpha _{i,j}\}\{id_{\overline{n}},\gamma
_{n}\alpha _{n-i+1}\gamma _{n}\}$ and $\Gamma _{2}=\{\alpha
_{i}^{2}\alpha _{j}^{2},\alpha _{i}\alpha _{j}^{2},\alpha _{i}^{2},\gamma
_{n}\alpha _{n-j+1}\gamma _{n}\}$.

Now, we consider the case $\text{rank}\delta =n-2$ and $\delta \in Par_{n}$. By Lemma \ref{lmm(bf4)},
there is exactly one $x\in dom\delta $ such that $x$ and $x\delta $ have different
parity and either $x$ is even and $x\delta \in \{1,n\}$ or $x\in \{1,n\}$.
Hence, there is an
even $a\in \overline{n}$ such that $\delta =\beta _{a}^{odd}$ or $\delta
=\left(
\begin{array}{ccccccccc}
1 & 2 & 3 & \cdots  & n-a+1 & n-a+2 & n-a+3 & \cdots  & n \\
a &  & n & \cdots  & a+2 &  & a-2 & \cdots  & 1%
\end{array}%
\right) =\alpha _{2}\beta _{a}^{odd}$ or $\delta =\beta _{a}^{even}$ or $%
\delta =\left(
\begin{array}{ccccccccc}
1 & \cdots  & a-2 & a-1 & a & a+1 & a+2 & \cdots  & n \\
n &  & n-a+3 &  & 1 &  & n-a+1 & \cdots  & 3%
\end{array}%
\right) =\beta _{a}^{even}\alpha _{2}$. Suppose that $a>\frac{n+1}{2}$. Then
we have $\beta _{a}^{odd}=\alpha _{2}\beta _{n-a+1}^{odd}\gamma _{n}$ and $%
\beta _{a}^{even}=\gamma _{n}\beta _{n-a+1}^{even}\alpha _{2}$, where $%
n-a+1<n-(\frac{n+1}{2})+1=\frac{n+1}{2}$.
\end{proof}

Lemma \ref{lmm4} shows $\left\langle J_{n}\right\rangle =\left\langle
G_{n}\right\rangle $. Thus, Proposition \ref{prop3} provides that $G_{n}$ is a
generating set for $\mathcal{FI}_{n}$.

\begin{corollary}\label{cor5}
$\mathcal{FI}_{n}=\left\langle G_{n}\right\rangle $.
\end{corollary}

It remains to show that $G_{n}$ is a generating set of minimal size.
It is useful to classify the partial injections in $\mathcal{FI}_{n}$
with \text{rank} $n-1$. For $1\leq i\leq \frac{n+1}{2}$, let $R_{i}:=\{\alpha \in
\mathcal{FI}_{n}:dom\alpha =\overline{n}\setminus \{i\}$ or $dom\alpha =%
\overline{n}\setminus \{n-i+1\}\}$. Clearly, $\bigcup \{R_{i}:1\leq i\leq
\frac{n+1}{2}\}=\{\alpha \in \mathcal{FI}_{n}:\text{rank}\alpha =n-1\}$. Moreover,
any generating set of $\mathcal{FI}_{n}$ contains elements form each $R_{i}$,
$1\leq i\leq \frac{n+1}{2}$, as the following lemma will show.

\begin{lemma}\label{lmm6}
Let $G\subseteq \mathcal{FI}_{n}$ with $\mathcal{FI}_{n}=\left\langle
G\right\rangle $. Then $G\cap R_{i}\neq \emptyset $ for all $i\in
\{1,2,\ldots ,\frac{n+1}{2}\}$.\qquad
\end{lemma}

\begin{proof}
Assume there is an $i\in \{1,2,\ldots ,\frac{n+1}{2}\}$ with $G\cap
R_{i}=\emptyset $. Then $\alpha _{i}\notin G$ and there are $g_{1},\ldots
,g_{s}\in G\setminus \{id_{\bar{n}}\}$ such that $\alpha _{i}=g_{1}\cdots g_{s}$,
where $g_{k}g_{k+1}\neq id_{\bar{n}}$ for $1\leq k<s$. This implies $\text{rank}g_{1}=n$,
i.e. $g_{1}=\gamma _{n}$, or $g_{1}\in R_{i}$. The latter one is not
possible. But $g_{1}=\gamma _{n}$ provides $\text{rank}g_{2}=n$, i.e. $g_{2}=\gamma
_{n}$, and thus $g_{1}g_{2}=id_{\bar{n}}$,\ or $g_{2}\in R_{i}$. Both are not
possible.
\end{proof}

Now, we are able to state the minimal size of a generating set of $\mathcal{FI%
}_{n}$. It will coincide with the size of $G_{n}$, which gives us the \text{rank}
of $\mathcal{FI}_{n}$.

\begin{proposition}\label{prop7}
Let $G$ be a generating set for $\mathcal{FI}_{n}$. Then $\left\vert G\right\vert \geq 5$ if $n=3$ and $\left\vert
G\right\vert \geq \frac{n-5}{2}+\left\lfloor \frac{n+6}{4}\right\rfloor
\left\lfloor \frac{n+7}{4}\right\rfloor $ if $n\geq 5$.
\end{proposition}

\begin{proof}
Since $\gamma _{n}$ and $id_{n}=\gamma _{n}\gamma _{n}$ are the only
transformations in $\mathcal{FI}_{n}$ with \text{rank} $n$, we can conclude that we
have at least one transformation with \text{rank} $n$ in $G$, namely $\gamma _{n}$.

Lemma \ref{lmm6} provides that there are at least $\left\lceil \frac{n}{4}%
\right\rceil $ transformations in $G\cap \bigcup \{R_{i}:1\leq i$ $\leq
\frac{n+1}{2}$, $i$ is odd$\}$, where $\left\lceil \frac{n}{4}\right\rceil $
$=\left\vert \{m\in \{1,\ldots ,\frac{n+1}{2}\}:m\text{ is odd}\}\right\vert
$. Moreover, there is at least one element in $G\cap R_{2}$, by Lemma \ref{lmm6}. If $%
n=3$ then we have at least $\left\lceil \frac{n}{4}\right\rceil +1=2$
elements in $G$ with \text{rank} $n-1$. Suppose that $n\geq 5$. If $\frac{n+1}{2}$ is
even then there is at least one element in $G\cap R_{\frac{n+1}{2}}$, by
Lemma \ref{lmm6}, too. If $n=5$ then we have at least $\left\lceil \frac{n}{4}%
\right\rceil +1=\left\lceil \frac{n}{4}\right\rceil +\frac{n-3}{2}$ elements
in $G$ with \text{rank} $n-1$. If $n=7$ then we have at least $\left\lceil \frac{n}{%
4}\right\rceil +1+1=\left\lceil \frac{n}{4}\right\rceil +\frac{n-3}{2}$
elements in $G$ with \text{rank} $n-1$. Let now $n\geq 9$. We consider the case
that $i\in \{4,\ldots ,\frac{n-1}{2}\}$ is even. Assume that $\left\vert
G\cap R_{i}\right\vert \leq 1$. It is easy to verify that $\left\vert
R_{i}\right\vert =16$, but on the other hand, for any $\alpha \in R_{i}$, we
have $\left\vert \left\langle \alpha ,\mathcal{FI}_{n}\setminus
R_{i}\right\rangle \cap R_{i}\right\vert \leq 8$, because $\left\langle
\alpha ,\mathcal{FI}_{n}\setminus R_{i}\right\rangle \cap R_{i}\subseteq
\{\alpha _{i},\alpha _{i}^{2},\alpha _{i}\gamma _{n},\alpha _{i}^{2}\gamma
_{n},\gamma _{n}\alpha _{i},\gamma _{n}\alpha _{i}^{2},\gamma _{n}\alpha
_{i}\gamma _{n},\gamma _{n}\alpha _{i}^{2}\gamma _{n}\}$. This implies that $%
R_{i}\nsubseteqq \left\langle G\right\rangle $, a contradiction. If $\frac{%
n+1}{2}$ is even then $\frac{n-7}{4}=\left\vert \{m\in \{4,\ldots ,\frac{n-1%
}{2}\}:m\text{ is even}\}\right\vert $, i.e. we have at least $1+1+2\left(
\frac{n-7}{4}\right) =\frac{n-3}{2}$ elements in $G\cap \bigcup
\{R_{i}:1\leq i$ $\leq \frac{n+1}{2}$, $i$ is even$\}$. If $\frac{n+1}{2}$
is odd then $\frac{n-5}{4}=\left\vert \{m\in \{4,\ldots ,\frac{n-1}{2}\}:m%
\text{ is even}\}\right\vert $, i.e. we have at least $1+2\left( \frac{n-5}{4%
}\right) =\frac{n-3}{2}$ elements in $G\cap \bigcup \{R_{i}:1\leq i$ $\leq
\frac{n+1}{2}$, $i$ is even$\}$. Altogether, we have at least $\left\lceil
\frac{n}{4}\right\rceil +\frac{n-3}{2}$ elements in $G$ with \text{rank} $n-1$,
whenever $n\geq 5$.

Assume that $\left\vert G\cap Par_{n}\right\vert <2\left\lfloor \frac{n+1}{4}%
\right\rfloor $, where $\left\lfloor \frac{n+1}{4}\right\rfloor =\left\vert
\left\{ m\in \{2,\ldots ,\frac{n+1}{2}\}:m\text{ is even}\right\}
\right\vert $. Note, there exists exactly one $x\in dom\alpha $ with $x$ and
x$\alpha $ have different parity, whenever $\alpha \in J_{n}\cap Par_{n}$, by
Lemma \ref{lmm(bf4)}. This
provides that there is an even $j\in \{2,\ldots ,\frac{n+1}{2}\}$ such that $%
j\alpha ,(n-j+1)\alpha \notin \{1,n\}$ for all $\alpha \in G$ or $1\alpha
,n\alpha \notin \{j,n-j+1\}$ for all $\alpha \in G$.

Suppose that $1\alpha ,n\alpha \notin \{j,n-j+1\}$ for all $\alpha \in G$.
In particular, this implies that $\beta _{j}^{odd}\notin G$. Hence, there
are $g_{1},\ldots ,g_{s}\in G\setminus \{id_{\bar{n}}\}$ such that $\beta
_{j}^{odd}=g_{1}\cdots g_{s}$. Since $1\beta _{j}^{odd}$ and $1$ have
different parity, we conclude that there is $k\in \{1,\ldots ,s\}$ such that
$g_{k}\in Par_{n}$. Without loss of generality, we can assume that $%
g_{r}\notin Par_{n}$ for $k<r\leq s$. By Lemma \ref{lmm(bf4)}, then there is an even $m\in \overline{n%
}$ such that $1g_{k}=m$ or $ng_{k}=m$. Since $1g_{k},ng_{k}\notin \{j,n-j+1\}
$, we conclude that $m\notin \{j,n-j+1\}$. Note that $\text{im}\beta _{j}^{odd}=%
\overline{n}\setminus \{j-1,j+1\}$. If $k=s$ then $\text{im}\beta
_{j}^{odd}=\text{im}(g_{1}\cdots g_{s})=\text{im}g_{k}=\overline{n}\setminus \{m-1,m+1\}\neq
\overline{n}\setminus \{j-1,j+1\}$, a contradiction. If $k<s$ then we put
\begin{equation*}
g:=g_{k+1}\cdots g_{s}\text{.}
\end{equation*}%
If $g=id_{n}$ then we get a contradiction by the previous arguments. If $%
g=\gamma _{n}$ then $\text{im}\beta _{j}^{odd}=\text{im}(g_{1}\cdots
g_{k}g)=\{n-m,n-m+2\}\neq \overline{n}\setminus \{j-1,j+1\}$ since $m\neq
n-j+1$, a contradiction. If $\text{rank}g=n-1$ then $\text{im}g=\overline{n}\setminus
\{j-1\}$ or $\text{im}g=\overline{n}\setminus \{j+1\}$. First, we consider the case $%
\text{im}g=\overline{n}\setminus \{j-1\}$. Since $j-1$ is odd, we can conclude that
$g\in \{\alpha _{j-1},\gamma _{n}\alpha _{j-1}\}$. Suppose that $g=\alpha
_{j-1}$. Since $\text{rank}g_{1}\cdots g_{k}g=\text{rank}g_{k}=n-2$, we have $%
\text{im}g_{k}\subseteq dom g$. This implies $j-1\in \{m-1,m+1\}$. Because of $j\neq
m$, we have $j-1=m+1$. This provides $m-1=j-3$, i.e. $j-3\notin
\text{im}g_{k}=\text{im}(g_{1}\cdots g_{k})=\text{im}(g_{1}\cdots g_{k}\alpha _{j-1})=\text{im}(g_{1}\cdots
g_{k}g)$, a contradiction. Suppose that $g=\gamma _{n}\alpha _{j-1}$.
Clearly, $dom g=\overline{n}\setminus \{n-j+2\}$. Then $\text{im}g_{k}\subseteq dom g$
provides $n-j+2\in \{m-1,m+1\}$. But $n-j+1\neq m$ gives $n-j+2=m-1$. Hence,
$m+1=n-j+4$ and we obtain $j-3=(n-j+4)\gamma _{n}\alpha _{j-1}$. Thus $%
j-3\notin \text{im}(g_{1}\cdots g_{k}g)$, a contradiction. Dually, we can treat the
case $\text{im}g=\overline{n}\setminus \{j+1\}$. If $\text{rank}g=n-2$ then $\text{im}g=\overline{%
n}\setminus \{j-1,j+1\}$. Since both $j-1$ and $j+1$ are odd, we can conclude
that $domg=\overline{n}\setminus \{j-1,j+1\}$ or $domg=\overline{n}\setminus
\{n-j,n-j+2\}$, since $g=\alpha _{j-1}\alpha _{j+1}$ or $g=\gamma _{n}\alpha
_{j-1}\alpha _{j+1}$. But because of $m\neq n-j+1$ and $m\neq j$, we have $m-1\neq n-j$
and $m-1\neq j-1$, respectively. This implies $img_{k}=\{m-1,m+1\}\neq
domg$, a contradiction.

Suppose that $j\alpha ,(n-j+1)\alpha \notin \{1,n\}$ for all $\alpha \in G$.
Then we conclude dually that there are $\delta _{1},\ldots ,\delta _{t}\in
G\setminus \{id_{\bar{n}}\}$ such that $\beta _{j}^{even}=\delta _{1}\cdots \delta
_{t}$ but $dom\delta _{1}\cdots \delta _{t}\neq dom$ $\beta _{j}^{even}$,
i.e. we obtain a contradiction, too.

Suppose that $n\geq 5$. By straightforward combinatorial calculations, one
obtains that there are exactly $\left\lfloor \frac{n}{4}\right\rfloor
\left\lfloor \frac{n+2}{4}\right\rfloor -1$ pairs $(i,j)$ of odd numbers $%
i,j\in \overline{n}$ with $4\leq j-i<n-1$ and $i\leq n-j+1$. Assume that
there are less than $\left\lfloor \frac{n}{4}\right\rfloor \left\lfloor
\frac{n+2}{4}\right\rfloor -1$ elements in $G\setminus Par_{n}$ with \text{rank} $%
n-2$. Note that for odd numbers $i,j\in \overline{n}$ with $4\leq j-i<n-1$
and $i\leq n-j+1$, it holds $4\leq (n-i+1)-(n-j+1)<n-1$ but $%
(n-j+1)>n-(n-i+1)+1$, whenever $i\neq n-j+1$. In particular, $i=n-j+1$
implies $j=n-i+1$. This justifies the existence of a pair $(i,j)$ of odd
numbers $i,j\in \overline{n}$ with $4\leq j-i<n-1$ and $i\leq n-j+1$ such
that $dom\alpha \neq \overline{n}\setminus \{i,j\}$ and $dom\alpha \neq
\overline{n}\setminus \{n-i+1,n-j+1\}$ for all $\alpha \in G$. In particular, $%
\alpha _{i,j}\notin G$ and there are $h_{1},\ldots ,h_{u}\in G\setminus
\{id_{\bar{n}}\}$ such that $\alpha _{i,j}=h_{1}\cdots h_{u}$. Note that each $%
\alpha \in \mathcal{FI}_{n}$ with $dom\alpha =\overline{n}\setminus
\{\upsilon \}$ for some odd $\upsilon \in \overline{n}$ is either
order-preserving or order-reversing. Assume that $rankh_{q}\geq n-1$ for all $q\in \{1,\ldots ,u\}$.
Since $\alpha _{i,j}$ is neither order-preserving nor order-reversing, there
is an even $v\in \overline{n}$ with $v\notin dom(h_{1}\cdots h_{u})$, a
contradiction. This implies the existence of a $k\in
\{1,\ldots ,u\}$ with $\text{rank}h_{k}=n-2$. Without loss of generality, we can assume that $rankh_{q}\geq
n-1 $ for $1\leq q<k$. In particular, an odd number is missing in $domh_{q}$, for $q\in
\{1,\ldots ,k\}$. Let $domh_{k}=\{s,t\}$. Since $%
h_{k}\in G$, we have $\{s,t\}\neq \{i,j\}$ as well as $\{s,t\}\neq
\{n-i+1,n-j+1\}$, i.e. $\{n-s+1,n-t+1\}\neq \{i,j\}$. Clearly, $%
dom h_{q}\subseteq \text{im}h_{q-1}$ for $1\leq q\leq k$. Therefore and since $h_{q}$ is
order-preserving or order-reversing for $1\leq q<k$, there is $h\in
\{id_{n},\gamma _{n}\}$ such that $h_{1}\cdots h_{k-1}h_{k}=hh_{k}$. Hence, $%
dom(h_{1}\cdots h_{k})=\{s,t\}$ or $dom(h_{1}\cdots h_{k})=\{n-s+1,n-t+1\}$.
Since $\text{rank}(h_{1}\cdots h_{k})=2=\text{rank}(h_{1}\cdots h_{u})$, we can conclude
that $dom\alpha _{i,j}=dom(h_{1}\cdots h_{u})=dom(h_{1}\cdots h_{k})\neq
\{i,j\}$, a contradiction.

Altogether, we have shown that $\left\vert G\right\vert \geq 1+2+2\left\lfloor \frac{3+1}{4}\right\rfloor =5$ if $n=3$ and
$\left\vert G\right\vert \geq 1+\left\lceil \frac{n}{4}\right\rceil +\frac{%
n-3}{2}+2\left\lfloor \frac{n+1}{4}\right\rfloor +\left\lfloor \frac{n}{4}%
\right\rfloor \left\lfloor \frac{n+2}{4}\right\rfloor -1$, whenever $n\geq 5$.
It is easy to verify that $\left\lceil \frac{n}{4}\right\rceil
+2\left\lfloor \frac{n+1}{4}\right\rfloor =n-\left\lfloor \frac{n}{4}%
\right\rfloor$ and $n-\left\lfloor \frac{n}{4}\right\rfloor +\left\lfloor
\frac{n}{4}\right\rfloor \left\lfloor \frac{n+2}{4}\right\rfloor
=\left\lfloor \frac{n+6}{4}\right\rfloor \left\lfloor \frac{n+7}{4}%
\right\rfloor -1$. Hence, $\left\vert G\right\vert
\geq \left\lfloor \frac{n+6}{4}\right\rfloor
\left\lfloor \frac{n+7}{4}\right\rfloor -1+\frac{n-3}{2}=\frac{n-5}{2}%
+\left\lfloor \frac{n+6}{4}\right\rfloor \left\lfloor \frac{n+7}{4}%
\right\rfloor $, whenever $n\geq 5$.
\end{proof}

It is easy to calculate that $\left\vert G_n\right\vert =1+\left\lceil \frac{n%
}{4}\right\rceil +\frac{n-3}{2}+2\left\lfloor \frac{n+1}{4}\right\rfloor
+\left\lfloor \frac{n}{4}\right\rfloor \left\lfloor \frac{n+2}{4}%
\right\rfloor -1=\frac{n-5}{2}+\left\lfloor \frac{n+6}{4}\right\rfloor
\left\lfloor \frac{n+7}{4}\right\rfloor $, whenever $n\geq 5$, and $%
\left\vert G_3\right\vert =5$. So, we get the following corollary.

\begin{corollary}\label{cor8}
$\text{rank}\mathcal{FI}_{3}=5$ and $\text{rank}\mathcal{FI}_{n}=\frac{n-5}{2}+\left\lfloor \frac{n+6%
}{4}\right\rfloor \left\lfloor \frac{n+7}{4}\right\rfloor $, whenever $n\geq
5$.
\end{corollary}

\section*{Acknowledgement}

The second author was supported by the Thailand Research Fund, Grant No. MRG6180296 and Research Fund
of Faculty of Science, Silpakorn University, Grant No. SRF-JRG-2561-08 through a visiting researcher fellowship.


{\small \sf
\noindent{\sc J\"{o}rg Koppitz},
Institute of Mathematics and Informatics,
Bulgarian Academy of Sciences,
Acad. G. Bonchev Str. Bl. 8,
1113 Sofia, Bulgaria;
email: koppitz@math.bas.bg

\medskip

\noindent{\sc Tiwadee Musunthia},
Department of Mathematics,
Faculty of Science,
Silpakorn University,
73000 Nakornpathom,
Thailand;
email: tiwadee.m@gmail.com

}

\end{document}